\documentclass[11pt]{amsart}

\usepackage{amsmath,amssymb,amsthm,mathtools}
\usepackage{enumitem}
\usepackage{microtype}
\usepackage[colorlinks=true,linkcolor=blue,citecolor=blue,urlcolor=blue]{hyperref}

\newtheorem{theorem}{Theorem}[section]
\newtheorem{proposition}[theorem]{Proposition}
\newtheorem{lemma}[theorem]{Lemma}
\newtheorem{corollary}[theorem]{Corollary}
\newtheorem{conjecture}[theorem]{Conjecture}
\theoremstyle{definition}
\newtheorem{definition}[theorem]{Definition}
\theoremstyle{remark}
\newtheorem{remark}[theorem]{Remark}

\DeclareMathOperator{\GL}{GL}
\DeclareMathOperator{\Sp}{Sp}
\DeclareMathOperator{\SL}{SL}
\DeclareMathOperator{\Hom}{Hom}
\DeclareMathOperator{\Irr}{Irr}
\DeclareMathOperator{\Cusp}{Cusp}
\DeclareMathOperator{\Nrd}{Nrd}
\DeclareMathOperator{\JL}{JL}
\DeclareMathOperator{\St}{St}
\DeclareMathOperator{\Ind}{Ind}

\DeclareMathOperator{\Sym}{Sym}

\newcommand{\C}{\mathbb C}

\newcommand{\cD}{\mathcal D}
\newcommand{\cB}{\mathcal B}
\newcommand{\ph}{\phi}
\newcommand{\bigboxtimes}{\mathop{\otimes}}

\title[Tempered symplectic distinction for quaternionic $\GL_n(D)$]
{Tempered symplectic distinction for quaternionic $\GL_n(D)$:\ A Conjecture of Prasad}

\author{Mahendra Kumar Verma}
\address{Department of Mathematics, Indian Institute of Technology Roorkee,
Uttarakhand 247667, India}
\email{mahendraverma@ma.iitr.ac.in}


\begin{document}

\begin{abstract}
Let $F$ be a non-Archimedean local field of characteristic zero, let $D/F$
be the quaternion division algebra, and put $G_n=\GL_n(D)$ and
$H_n=\Sp_n(D)$. We prove Prasad's revised conjecture on tempered
representations with symplectic period, which may be stated as follows. If $\rho$ is a unitary
$H_r$-distinguished supercuspidal representation and $\delta_{i}(\rho)$ is the generalised Steinberg representation associated to $\rho$ , set
\[
   \cB_0(\rho)=\rho,
   \qquad
   \cB_m(\rho)=\delta_{m+1}(\rho)\times\delta_m(\rho),
   \quad m\ge1.
\]
Then an irreducible tempered representation $\pi$ of $G_n$ is
$H_n$-distinguished if and only if
\[
   \pi\simeq\prod_i\cB_{m_i}(\rho_i).
\]
\end{abstract}

\maketitle

\section{Introduction}

Let $F$ be a non-Archimedean local field of characteristic zero and let $D$
be the quaternion division algebra over $F$. For $n\geq 1$, let
$
G_n=\GL_n(D),
$
the inner form of the split general linear group $\GL_{2n}(F)$, and let
$H_n=\Sp_n(D),
$
the corresponding quaternionic symplectic group, which is the unique inner form of
the split symplectic group $\Sp_{2n}(F)$. An irreducible admissible
representation $\pi$ of $G_n$ is said to be $H_n$-distinguished if
\[
\Hom_{H_n}(\pi,\C)\neq 0.
\]
In this case, we also say that $\pi$ admits a symplectic model. The
symplectic model is unique; see \cite{Verma2018}.

Distinguished representations play a crucial role in studying automorphic forms and representations,
particularly in the context of the Langlands program. In the global context, distinguished
representations help to analyze periods of automorphic forms over subgroups, providing insights into special values of $L$-functions and other arithmetic properties. These representations are also fundamental in
harmonic analysis on $p$-adic groups and their homogeneous spaces.

When $D$ is replaced by $F$, the problem becomes that of representations of $\GL_{2n}(F)$ distinguished by $\Sp_{2n}(F)$. This question was initiated by Heumos and Rallis \cite{HeumosRallis1990} and was further studied by Offen and Sayag, who classified the unitary representations admitting a symplectic period and investigated their relation with Klyachko models and global mixed periods \cite{OffenSayag2007,OffenSayag2008Global,OffenSayag2008Uniqueness}. The Klyachko models interpolate between two extreme cases: the Whittaker model at one end and the symplectic model at the other. Thus the symplectic model can be viewed as one of the extreme cases in the family of Klyachko models.

Our general goal is to understand and classify irreducible representations of \(\GL_n(D)\) which admit a symplectic model. A classification for all irreducible representations appears to be a difficult problem. The  next natural  class to consider is the unitary dual. In \cite{Verma2018}, a conjecture of Prasad was formulated for distinguished unitary representations, but the conjecture is not correct as stated. In this paper we therefore restrict ourselves to a smaller, but still important, class of representations, namely the tempered representations.

The tempered case also has a global motivation. Let \(F\) now be a number field. A non-zero global symplectic period gives local distinction at every place. For cuspidal automorphic representations of general linear groups, the generalized Ramanujan conjecture predicts that the local components are tempered. Thus a classification of distinguished tempered representations describes the local representations which can occur at a tempered place of a cuspidal automorphic representation having a symplectic period. For inner forms, this question is also naturally related to the local and global Jacquet--Langlands correspondence.

In \cite{Verma2018} we studied symplectic periods for the quaternionic pair
$G_n/H_n$. We proved multiplicity one, studied distinction for $G_2$, constructed
supercuspidal distinguished representations, and formulated local and global
questions. Sharma and Verma later studied  ladder and
unitary representations and  some further cases \cite{SharmaVerma2025,SharmaVerma2026}; the
corresponding split arguments may be found in
\cite{Mitra2014,MitraOffenSayag2017,OffenSayag2007}. Hariom Sharma has also
studied symplectic models for Zelevinsky modules of $\GL_n(D)$ 
\cite{SharmaZelevinsky2026}.

The discrete-series conjecture formulated by D.~Prasad was recorded in
\cite[Conjecture~7.1(1)]{Verma2018}. It is now known: a discrete-series representation of $G_n$ is
$H_n$-distinguished if and only if it is supercuspidal and its
Jacquet--Langlands transfer is non-supercuspidal; see
\cite{MSSS2026,Matringe2026}. The unitary statement recorded in
\cite[Conjecture~7.1(2)]{Verma2018} is not correct as stated. In August 2026
Prasad proposed a revised statement for tempered representations
\cite{dprasad}. This leads to following  main result considered
in this paper.

Let $\cD_r$ denote the set of unitary $H_r$-distinguished supercuspidal
representations of $G_r$. For $\rho\in\cD_r$, define
\[
\cB_0(\rho)=\rho,
\qquad
\cB_m(\rho)
=
\delta_{m+1}(\rho)\times\delta_m(\rho),
\qquad m\geq 1.
\]

\begin{theorem}\label{conj:tempered-classification}
	Let $\pi$ be an irreducible tempered representation of $G_n$. Then $\pi$ is
	$H_n$-distinguished if and only if there exist integers
	$m_1,\ldots,m_k\geq 0$ and unitary $H_{r_i}$-distinguished supercuspidal
	representations
	\[
	\rho_i\in\cD_{r_i},
	\qquad 1\leq i\leq k,
	\]
	such that
	\[
	\pi
	\simeq
	\cB_{m_1}(\rho_1)
	\times\cdots\times
	\cB_{m_k}(\rho_k),
	\]
	with
	\[
	\sum_{i=1}^k (2m_i+1)r_i=n.
	\]
	Equivalently, every irreducible tempered $H_n$-distinguished representation
	is obtained by parabolic induction from blocks of the form
	\[
	\rho
	\qquad\text{or}\qquad
	\delta_{m+1}(\rho)\times\delta_m(\rho),
	\]
	where $\rho$ is a unitary distinguished supercuspidal representation.
\end{theorem}
The revised conjecture says that a tempered representation is distinguished
if and only if it is a product of  representations $ \cB_m(\rho).$ Thus, on each
distinguished supercuspidal line, the lengths of the discrete-series factors
are divided into the singleton $\{1\}$ and pairs $\{m,m+1\}$. The low-rank
examples in \cite{SharmaVerma2025,SharmaVerma2026}, with the split analogues
in \cite{Mitra2014,MitraOffenSayag2017}, already show why a statement in
terms of individual distinguished discrete-series factors cannot be correct.

The blocks have a direct parameter interpretation. If $\rho$ is distinguished, its
parameter is of the form $\tau\otimes S_2$. The block
$\delta_{m+1}(\rho)\times\delta_m(\rho)$ corresponds to
\[
 \tau\otimes(S_{2m+2}\oplus S_{2m})
 =\tau\otimes(S_{2m+1}\otimes S_2),
\]
by the Clebsch--Gordan formula.

We prove first that each $\cB_m(\rho)$ is distinguished, using the closed
orbit for a maximal parabolic and the Jacquet module of
$\delta_{m+1}(\rho)$. The hereditary property then gives the existence
direction of the conjecture.

For the converse, write the tempered representation in the form
\[
   \pi=\delta_L(\rho)\times\sigma,
\]
where $L$ is maximal on the $\rho$-line. We use the Mackey filtration for
the maximal parabolic $P_{Ld,n-Ld}$. First we prove that $\rho$ must be
distinguished. If $L>1$, the only orbit which can contribute has
$r=(L-1)d$. It forces a factor $\delta_{L-1}(\rho)$ in $\sigma$, and the
remaining representation is distinguished. Here we reconstructed the induced representation from its Jacquet constituent. This yields the rank reduction
\[
   \pi\simeq\cB_{L-1}(\rho)\times\pi_0.
\]
Induction proves the converse.

The paper is organized as follows. We fixes notation
and recalls the used form of the geometric lemma in Section~\ref{sec:prelim}. Section~\ref{sec:blocks}
introduces the distinguished supercuspidal lines and the blocks
$\cB_m(\rho)$. Section~\ref{sec:conj} gives the block, adjacent-length and
parameter forms of Prasad's revised conjecture. Section~\ref{sec:Bm-general}
proves the existence of the elementary blocks and records the rank-five
specialization. Section~\ref{sec:converse} proves the converse by rank
reduction.

\pagebreak
\begin{center}
	\large {\textbf{Acknowledgment}}
\end{center}
The author would like to thank Dipendra Prasad for informing him about the revised conjecture considered in this paper, and for several useful discussions, comments, and suggestions during the preparation of the manuscript.

 AI  was used for language revision in the prepration of the article.

\section{Notation and preliminaries}\label{sec:prelim}

\subsection{The symmetric pair}

Let $\overline{x}$ denote the canonical involution on $D$. Let $J_n$ be the
anti-diagonal matrix with all anti-diagonal entries equal to $1$. We define
$H_n$ as
\[
 H_n
 =
 \left\{
 g\in G_n:
 gJ_n\,{}^t\!\overline{g}=J_n
 \right\}.
\]
We write
$\Irr(G_n)
$
for the set of isomorphism classes of irreducible smooth complex
representations of $G_n$.

Let $
 \nu(g)=|\Nrd(g)|_F,\qquad g\in G_n
$ denote the character of $G_n$ given by the absolute value of the reduced norm.
Since every element of $H_n$ has reduced norm one,
 $\nu|_{H_n}=1.$

Consequently,
\begin{equation}\label{eq:twist-dist}
 \Hom_{H_n}(\pi,\C)\neq0
 \quad\Longleftrightarrow\quad
 \Hom_{H_n}(\nu^s\pi,\C)\neq0
\end{equation}
for every $s\in\C$.

All parabolic inductions in this paper are normalized. If
$\pi_i\in\Irr(G_{n_i})$, we write
$
 \pi_1\times\cdots\times\pi_t
$
for the normalized induction to $G_{n_1+\cdots+n_t}$.

\subsection{Cuspidal lines and discrete series}

We denote by
$  \Cusp(G_r)\subset\Irr(G_r)$
the set of irreducible supercuspidal representations of $G_r$. Let
$\rho\in\Cusp(G_r)$. We write
$
 \nu_\rho=\nu^{s(\rho)},
$
where $s(\rho)\in\{1,2\}$ is the reducibility number of the cuspidal line of
$\rho$; see \cite{DeligneKazhdanVigneras1984,Tadic1990,SharmaVerma2026}.

A segment on the $\rho$-cuspidal line is denoted by
\[
[a,b]_\rho
=
\{\nu_\rho^a\rho,\nu_\rho^{a+1}\rho,\ldots,\nu_\rho^b\rho\},
\qquad a\leq b .
\]
We denote by $Q([a,b]_\rho)$ the unique irreducible quotient of the
normalized parabolically induced representation
\[
\nu_\rho^a\rho
\times
\nu_\rho^{a+1}\rho
\times\cdots\times
\nu_\rho^b\rho.
\]
It is the essentially square-integrable representation associated with the
segment $[a,b]_\rho$.

We will mainly work with unitary discrete-series representations. Let $\rho$
be a unitary supercuspidal representation. For $m\geq 1$, define
$\delta_m(\rho)$ ( also denoted by   $\St_m(\rho)$ in the litreture, called generalised Steinberg ) to be the unique irreducible quotient of
\[
\nu^{(1-m)s(\rho)/2}\rho
\times
\nu^{(3-m)s(\rho)/2}\rho
\times\cdots\times
\nu^{(m-1)s(\rho)/2}\rho.
\]
Equivalently,
$
\delta_m(\rho)
=
Q\left(
\left[\frac{1-m}{2},\frac{m-1}{2}\right]_\rho
\right).
$
Thus $\delta_m(\rho)$ is the unitary discrete-series representation
associated with the centered segment of length $m$ on the $\rho$-cuspidal
line. In particular,
$
\delta_1(\rho)=\rho.
$

When $s(\rho)=2$, we have
\[
\delta_m(\rho)
=
\operatorname{quot}\left(
\nu^{1-m}\rho
\times
\nu^{3-m}\rho
\times\cdots\times
\nu^{m-1}\rho
\right),
\]
where $\operatorname{quot}$ denotes the unique irreducible quotient.

Equivalently, in $\nu_\rho$-notation its segment is centered at zero:
\[
 \left[-\frac{m-1}{2},\frac{m-1}{2}\right]_\rho.
\]

\begin{lemma}[Two-block Jacquet module]\label{lem:two-block-jacquet}
Let $\rho$ be a unitary supercuspidal representation of $G_d$, put
$s=s(\rho)$, and let $1\le a\le m-1$. Then
\begin{equation}\label{eq:general-two-block-jacquet}
\begin{split}
 &r_{(ad,(m-a)d),(md)}\bigl(\delta_m(\rho)\bigr)\\
 &\qquad=
 \nu^{s(m-a)/2}\delta_a(\rho)
 \otimes
 \nu^{-sa/2}\delta_{m-a}(\rho).
\end{split}
\end{equation}
More generally, a Jacquet module of $\delta_m(\rho)$ is zero unless every
block size in the chosen Levi is divisible by $d$.
\end{lemma}

\begin{proof}
The centered segment defining $\delta_m(\rho)$ is
$
 \left[-\frac{m-1}{2},\frac{m-1}{2}\right]_\rho.
$
For the Levi of type $(ad,(m-a)d)$, the segment Jacquet-module formula cuts
this segment into consecutive pieces of lengths $a$ and $m-a$. Their
centers are $(m-a)/2$ and $-a/2$, respectively. Since
$\nu_\rho=\nu^s$, this gives \eqref{eq:general-two-block-jacquet}. The
vanishing assertion follows from the fact that every cuspidal atom on the
$\rho$-line has rank $d$; see \cite[\S2.4]{SharmaVerma2026} and
\cite{DeligneKazhdanVigneras1984,Tadic1990}.
\end{proof}

\subsection{Tempered representations}

Every irreducible tempered representation of $G_n$ can be written, uniquely
 in the form
\begin{equation}\label{eq:temp-decomp}
 \pi=\delta_1\times\cdots\times\delta_t,
\end{equation}
where every $\delta_i$ is a unitary discrete-series representation of some
$G_{n_i}$ and
$
 n_1+\cdots+n_t=n.
$
For inner forms of general linear groups the product in
\eqref{eq:temp-decomp} is irreducible; see \cite{Tadic1990,Tadic2015,LapidMinguez2016}.

We use repeatedly the hereditary property of the symplectic model:
if every $\sigma_i$ is $H_{n_i}$-distinguished, then
$ \sigma_1\times\cdots\times\sigma_k
$
is $H_{n_1+\cdots+n_k}$-distinguished; see
\cite[Theorem~4.1]{SharmaVerma2026}. This is the hereditary property proved
in the symmetric-space setting in \cite[Proposition~7.1]{Offen2017}; see also
\cite[Theorem~5.4]{MatringeOffenYang2024}.

\subsection{The geometric lemma, good orbits, and their combinatorial data}
\label{subsec:good-orbit-general}

We recall the definition of a good orbit and the criterion used below. The
filtration comes from the geometric lemma of Bernstein--Zelevinsky
\cite[Theorem~5.2]{BernsteinZelevinsky1977}; the orbit calculation used here
is given in \cite[\S\S3--4]{SharmaVerma2026}, following the symmetric-space
analysis of \cite{Offen2017}.

Let $
   P=P_\alpha=M_\alpha U_\alpha\subset G_n,
   \qquad
   \alpha=(n_1,\ldots,n_t),
$
be a standard parabolic subgroup, and let
$
   M_\alpha\simeq G_{n_1}\times\cdots\times G_{n_t}.
$
For a finite-length representation
$
   \sigma=\sigma_1\otimes\cdots\otimes\sigma_t
   \quad\text{of }M_\alpha,
$
we write
\[
   I_P(\sigma)=\Ind_P^{G_n}(\sigma),
\]
where induction is normalized.

We write
 $  X_n=G_n/H_n.$
The double cosets
$  P\backslash G_n/H_n$
are the same as the $P$-orbits in $X_n$. Choose representatives
$ y_1,\ldots,y_N\in G_n$ in an order such that
\[
   \bigcup_{j\leq i}Py_jH_n
\]
is open for every $i$. If $\mathcal V$ is the space of $I_P(\sigma)$, let
\[
   \mathcal V_i
   =
   \{f\in\mathcal V:
      \operatorname{Supp}(f)
      \subset
      \bigcup_{j\leq i}Py_jH_n
   \}.
\]
Then
\begin{equation}\label{eq:mackey-good-filtration}
   0=\mathcal V_0
   \subset\mathcal V_1
   \subset\cdots\subset
   \mathcal V_N=\mathcal V
\end{equation}
is an $H_n$-invariant filtration. For the orbit represented by $y_i$, put
\[
   H_{y_i}=H_n\cap y_i^{-1}Py_i.
\]
The corresponding graded piece is
\begin{equation}\label{eq:mackey-graded-piece}
   \mathcal V_i/\mathcal V_{i-1}
   \simeq
   \operatorname{ind}_{H_{y_i}}^{H_n}
   \left(
      \left.
      (\delta_P^{1/2}\sigma)^{y_i}
      \right|_{H_{y_i}}
   \right),
\end{equation}
where the induction on the right is the usual compactly supported
non-normalized induction.

\begin{definition}[Good orbit]\label{def:good-orbit}
The orbit
\[
   Py_iH_n
   \quad\text{or, equivalently,}\quad
   P\cdot(y_iH_n)\subset X_n
\]
is called a good orbit for $\sigma$ if
\begin{equation}\label{eq:def-good-orbit}
   \Hom_{H_n}
   \left(
      \mathcal V_i/\mathcal V_{i-1},
      \C
   \right)
   \neq0.
\end{equation}
\end{definition}

Thus ``good'' refers to one graded piece of the Mackey geometric lemma
filtration.
\begin{lemma}[Distinction forces a good orbit]
\label{lem:dist-implies-good}
If
$\Hom_{H_n}\bigl(I_P(\sigma),\C\bigr)\neq0,$
then at least one orbit in $P\backslash G_n/H_n$ is good for $\sigma$.
\end{lemma}

\begin{proof}
Let
\[
   0\neq\ell\in\Hom_{H_n}(\mathcal V,\C).
\]
Choose the smallest $i$ for which
$ \ell|_{\mathcal V_i}\neq0.$ Then $\ell$ vanishes on $\mathcal V_{i-1}$ and therefore induces a nonzero
$H_n$-invariant functional on
$\mathcal V_i/\mathcal V_{i-1}$.
\end{proof}

\begin{remark}[A good orbit is only a necessary condition in general]
\label{rem:good-not-sufficient}
The converse of Lemma~\ref{lem:dist-implies-good} is not automatic. A good
orbit gives an invariant functional on a graded piece, and such a functional
need not extend to the whole induced representation. One situation in
which the direction can be reversed is a closed orbit occurring as the final
quotient of \eqref{eq:mackey-good-filtration}: a functional on that quotient
pulls back to the whole representation.
\end{remark}

\subsubsection{Refined orbit data}

For later use, we express goodness in the combinatorial form used below.
A $P$-orbit determines a standard Levi subgroup
\[
   L=M_\beta\subset M_\alpha,
\]
where
$
   \beta=(\beta_1,\ldots,\beta_t),
   \qquad
   \beta_i=(n_{i,1},\ldots,n_{i,k_i})
$
is a refinement of $n_i$. Introduce the index set
\[
   \mathfrak I
   =
   \{(i,j):
     1\leq i\leq t,\ 1\leq j\leq k_i\}.
\]
We use two orders on $\mathfrak I$: we write
$ (i,j)\prec(i',j')$
for the lexicographic order, and define partial order
\[
	(i,j)\ll(i',j')
	\quad\Longleftrightarrow\quad
	i<i'.
\]
The second relation remembers only the original block of $M_\alpha$.

The orbit also determines an involution
\[
   \tau:\mathfrak I\longrightarrow\mathfrak I
\]
which preserves the sizes of the refined blocks:
$ n_{i,j}=n_{\tau(i,j)}.$

The orbit condition gives the ordering restriction
\begin{equation}\label{eq:ordering-restriction}
	\tau(i,j+1)\ll\tau(i,j)
	\qquad
	(1\leq j<k_i).
\end{equation}
In particular, as one moves from one Jacquet piece to the next inside a
fixed original factor, the original block containing its partner moves
strictly to the left. Consequently, at most one piece from a given original
factor can be fixed by $\tau$.

\begin{equation}\label{eq:jacquet-constituent-good}
{\rm If~~~~} \qquad	\eta
	=
	\bigboxtimes_{(i,j)\in\mathfrak I}\eta_{i,j}
\end{equation} 
is an irreducible constituent of \(r_{L,M_\alpha}(\sigma)\), then the modulus characters arising from the orbit and from normalized Jacquet functors give a single \(\nu\)-twist on each pair of factors exchanged by the orbit involution. Consequently, for \((i,j)\prec\tau(i,j)\), the good-orbit condition requires
\[
\eta_{i,j}\simeq \nu\,\eta_{\tau(i,j)}.
\]

\begin{proposition}[Representation-theoretic good-orbit criterion]
\label{prop:good-orbit-general}
For a fixed orbit and its associated data $(\beta,\tau)$, the corresponding
graded piece can contribute an $H_n$-invariant functional only if
$r_{L,M_\alpha}(\sigma)$ has an irreducible constituent
\eqref{eq:jacquet-constituent-good} satisfying:
\begin{enumerate}[label=\textup{(\roman*)}]
\item if
$(i,j)\prec\tau(i,j),$
then
\begin{equation}\label{eq:nu-pair-good}
	\eta_{i,j}\simeq \nu\,\eta_{\tau(i,j)}.
\end{equation}
\item if
$ \tau(i,j)=(i,j),$
then
\begin{equation}\label{eq:fixed-good}
	\Hom_{H_{n_{i,j}}}(\eta_{i,j},\C)\neq0.
\end{equation}
\end{enumerate}
If the normalized Jacquet module
$r_{L,M_\alpha}(\sigma)$ is itself irreducible, then these two conditions,
together with \eqref{eq:ordering-restriction}, are equivalent to the orbit
being good.
\end{proposition}

\begin{proof}
For the orbit represented by $x$, Proposition~4.1 of \cite{SharmaVerma2026}, following
\cite[Proposition~4.1]{Offen2017}, identifies the Hom-space of the graded piece with
\[
   \Hom_{L_x}
   \left(
      r_{L,M_\alpha}(\sigma),
      \delta_{Q_x}\delta_Q^{-1/2}
   \right).
\]
Proposition~4.2 of \cite{SharmaVerma2026}, using the orbit description of
\cite{Offen2017}, evaluates this Hom-space on an
irreducible tensor representation and gives exactly
\eqref{eq:nu-pair-good} and \eqref{eq:fixed-good}. The ordering condition
\eqref{eq:ordering-restriction} is the orbit condition from
\cite[\S3.3]{SharmaVerma2026}; compare \cite[Lemma~3.2]{Offen2017}. If the
Jacquet module is irreducible, no
passage to another constituent is necessary, and the criterion is an
equivalence.
\end{proof}

\subsubsection{Equivalent segment data}

For the tempered representations considered in this paper, the inducing
factors are essentially square-integrable segment representations, and the
criterion can be written in terms of segments.

Suppose
\[
   \sigma
   =
   Q(\Delta_1)\otimes\cdots\otimes Q(\Delta_t),
\]
where
$
   \Delta_i=[a_i,b_i]_{\rho_i}
$
is a segment on the supercuspidal line of
$\rho_i\in\Cusp(G_{d_i})$. For the refinement
$
   \beta_i=(n_{i,1},\ldots,n_{i,k_i}),
$
the Jacquet module is zero unless
$
   n_{i,j}=d_i\ell_{i,j}
$
for positive integers $\ell_{i,j}$ summing to the length of $\Delta_i$.
When it is nonzero, $\Delta_i$ is cut into consecutive subsegments
\begin{equation}\label{eq:segment-good-decomposition}
   \Delta_i
   =
   \Delta_{i,1}\sqcup\cdots\sqcup\Delta_{i,k_i},
\end{equation}
starting from the right endpoint, and
\begin{equation}\label{eq:segment-good-jacquet}
   r_{\beta_i,(n_i)}Q(\Delta_i)
   =
   Q(\Delta_{i,1})
   \otimes\cdots\otimes
   Q(\Delta_{i,k_i}).
\end{equation}
Therefore
$r_{L,M_\alpha}(\sigma)
   =
   \bigboxtimes_{(i,j)\in\mathfrak I}
   Q(\Delta_{i,j})
$
is irreducible.

\begin{definition}[Combinatorially good segment data]
\label{def:combinatorial-good}
The segment decomposition
\eqref{eq:segment-good-decomposition}, together with the orbit involution
$\tau$, is called \emph{combinatorially good} if:
\begin{enumerate}[label=\textup{(C\arabic*)}]
\item
$ \tau(i,j+1)\ll\tau(i,j)
   \qquad(1\leq j<k_i);$

\item for every two-cycle, oriented by
$(i,j)\prec\tau(i,j)$,
$
   Q(\Delta_{i,j})
   \simeq
   \nu\,Q(\Delta_{\tau(i,j)});
$
\item for every fixed point,
$
   Q(\Delta_{i,j})
   \quad\text{is }H_{n_{i,j}}\text{-distinguished}.
$
\end{enumerate}
\end{definition}

\begin{corollary}[Good orbit $\Longleftrightarrow$ combinatorially good data]
\label{cor:good-combinatorial}
Let
\[
   \sigma
   =
   Q(\Delta_1)\otimes\cdots\otimes Q(\Delta_t).
\]
For a fixed $P_\alpha$-orbit whose Jacquet refinement is nonzero, the orbit
is good for $\sigma$ if and only if its associated segment data are
combinatorially good in the sense of
Definition~\ref{def:combinatorial-good}.
\end{corollary}

\begin{proof}
By \eqref{eq:segment-good-jacquet}, the relevant normalized Jacquet module is
the single irreducible tensor product
\[
   \bigboxtimes_{(i,j)}
   Q(\Delta_{i,j}).
\]
The result is therefore the irreducible case of
Proposition~\ref{prop:good-orbit-general}.
\end{proof}

The fixed-point condition can be made still more explicit. Since
$Q(\Delta_{i,j})$ is essentially square-integrable,
Theorem~\ref{thm:ds}, together with twist invariance
\eqref{eq:twist-dist}, gives
\[
   Q(\Delta_{i,j})\text{ fixed and distinguished}
   \quad\Longrightarrow\quad
   \Delta_{i,j}\text{ has length }1,
\]
and the underlying supercuspidal representation is distinguished, up to an
unramified twist.

We can conclude that  every refined segment piece is either
	paired with a same-rank piece differing by one $\nu\text{-twist,}$
	or is a distinguished singleton fixed by $\tau$
subject always to the strict ordering rule
\[
   \tau(i,j+1)\ll\tau(i,j).
\]

If the original supercuspidal line has no symplectic model, fixed points are
impossible. In that case this reduces to the fixed-point-free good
decomposition used in \cite[Definition~5.1]{SharmaVerma2026}; this is the quaternionic form of
the decomposition used in \cite{MitraOffenSayag2017}. On a distinguished supercuspidal line, fixed singleton pieces are allowed.

\subsubsection{The maximal-parabolic case}
For our purpose, we will specialize above discussion to the maximal paraboilic. 
Let
$ P=P_{k,n-k},
   \qquad
   k\leq n-k,$
and
 $  \sigma=\lambda_1\otimes\lambda_2.
$
Then $H_n$-orbits are indexed by
$  0\leq r\leq k.$

For the orbit $O_r$, the refined Levi is obtained from
\[
   k=r+(k-r),
   \qquad
   n-k=(n-k-r)+r.
\]
The two rank-$r$ pieces are exchanged by $\tau$, while the pieces of ranks
$k-r$ and $n-k-r$ are fixed. Consequently the abstract criterion above
becomes
\begin{equation}\label{eq:maximal-good-general}
\begin{split}
   &\Hom_{M_{k,r}}
   \left(
      r_{(r,k-r),(k)}(\lambda_1)
      \otimes
      r_{(n-k-r,r),(n-k)}(\lambda_2),
      \nu
   \right)\neq0,\\
   &\qquad
   M_{k,r}
   \simeq
   \Delta(G_r\times G_r)
   \times H_{k-r}
   \times H_{n-k-r}.
\end{split}
\end{equation}
This is the maximal-parabolic form of
\cite[Corollary~4.1]{SharmaVerma2026}, obtained from the geometric lemma
\cite{BernsteinZelevinsky1977} and the orbit calculation of \cite{Offen2017}.
Formula
\eqref{eq:maximal-good-general} is the maximal-parabolic form used later in
the proof of converse.

\section{Distinguished supercuspidal lines and elementary blocks}
\label{sec:blocks}

Let $\cD_r$ denote the set of unitary supercuspidal representations
$\rho$ of $G_r$ satisfying
\[
 \Hom_{H_r}(\rho,\C)\neq0.
\]
Set
\[
 \cD=\bigcup_{r\geq1}\cD_r.
\]

The discrete-series case is the starting point for the tempered conjecture.
We recall first the conjecture formulated by D.~Prasad and recorded in
\cite[Conjecture~7.1(1)]{Verma2018}.

\begin{conjecture}[Prasad's discrete-series conjecture]\label{conj:prasad-ds}
Let $\delta$ be an irreducible discrete-series representation of $G_n$ and
write its Weil--Deligne parameter in the form
\[
   \phi_\delta=\tau\otimes S_a,
\]
where $\tau$ is an irreducible representation of $W_F$ and $S_a$ is the
$a$-dimensional irreducible representation of $\SL_2(\C)$. Then $\delta$
is $H_n$-distinguished if and only if $\delta$ is supercuspidal,
$\tau$ is odd-dimensional, and
  $ a=2.$

\end{conjecture}

Thus the parameter predicted for a distinguished supercuspidal
representation is directly of the form
\begin{equation}\label{eq:dist-cusp-param}
   \phi_\rho=\tau_\rho\otimes S_2,
   \qquad
   \dim(\tau_\rho)\ \text{odd}.
\end{equation}

This conjecture is now known in the form needed here.

\begin{theorem}[Discrete-series distinction]\label{thm:ds}
Let $\delta$ be an irreducible discrete-series representation of $G_n$.
Then
\[
 \Hom_{H_n}(\delta,\C)\neq0
\]
if and only if $\delta$ is supercuspidal and its Jacquet--Langlands transfer
$\JL(\delta)$ to $\GL_{2n}(F)$ is non-supercuspidal.
Equivalently, the distinguished discrete-series representations are exactly
the supercuspidal representations predicted by
Conjecture~\ref{conj:prasad-ds}. See \cite[Corollary~1.4]{Matringe2026}; and
\cite{MSSS2026}.
\end{theorem}

For $\rho\in\cD_r$, the Jacquet--Langlands transfer is therefore of the form
\[
 \JL(\rho)=\St_2(\sigma_\rho)
\]
for a unitary supercuspidal representation $\sigma_\rho$ of $\GL_r(F)$.
Equivalently, \eqref{eq:dist-cusp-param} holds. In particular,
 $s(\rho)=2,$
and $r$ is odd.

\begin{definition}[Elementary tempered blocks]\label{def:block}
Let $\rho\in\cD_r$. Put
\[
 \cB_0(\rho)=\rho.
\]
For $m\geq1$, define
\begin{equation}\label{eq:block}
 \cB_m(\rho)
 :=
 \delta_{m+1}(\rho)\times\delta_m(\rho).
\end{equation}
Thus $\cB_m(\rho)$ is a representation of
$ G_{(2m+1)r}.$

\end{definition}

\begin{lemma}\label{lem:block-temp}
For every $\rho\in\cD_r$ and $m\geq1$, the representation
$\cB_m(\rho)$ is irreducible and tempered.
\end{lemma}

\begin{proof}
Both factors in \eqref{eq:block} are unitary discrete-series
representations. Their defining centered segments are nested, hence are not
linked. The irreducibility criterion for products of segment
representations therefore gives the irreducibility of
$\delta_{m+1}(\rho)\times\delta_m(\rho)$. Since it is a normalized
parabolic induction of unitary discrete-series representations, it is
tempered, see \cite[Corollary 6.3]{Tadic2015}.
\end{proof}

The blocks $\cB_m(\rho)$ should be viewed as the tempered replacement for a
single non-supercuspidal discrete-series representation. Theorem
\ref{thm:ds} says that $\delta_m(\rho)$ cannot be distinguished for $m>1$;
the conjecture below predicts that two consecutive lengths can nevertheless
survive together.

\section{Prasad's revised conjecture for tempered symplectic distinction}\label{sec:conj}

The unitary conjecture formulated by D.~Prasad and recorded in
\cite[Conjecture~7.1(2)]{Verma2018} is not correct as stated there. In
recent email correspondence Prasad has revised the assertion and proposed a
classification for distinguished tempered representations
\cite{dprasad}. We use the name Prasad's conjecture for distinguished
tempered representations of $\GL_n(D)$ with symplectic period for this
revised assertion. Its block form is the following.

\begin{conjecture}[Prasad's revised tempered distinction conjecture: block form]
\label{conj:block}
Let $\pi\in\Irr(G_n)$ be tempered. Then
$
 \Hom_{H_n}(\pi,\C)\neq0
$
if and only if
\[
 \pi
 \simeq
 \cB_{m_1}(\rho_1)
 \times\cdots\times
 \cB_{m_t}(\rho_t),
\]
where $\rho_i\in\cD$ and $m_i\geq0$ for every $i$.
\end{conjecture}

Thus every connected piece is either a distinguished supercuspidal singleton
or an adjacent pair of discrete-series representations:
\[
\rho
\qquad\text{or}\qquad
\delta_{m+1}(\rho)\times\delta_m(\rho).
\]
\begin{remark}
Section~\ref{sec:Bm-general} proves the ``if'' direction of
Conjecture~\ref{conj:block} in arbitrary rank: every elementary block
$\cB_m(\rho)$ is distinguished, and arbitrary products of such blocks are
distinguished by the hereditary property. The substantive remaining
content is the converse assertion that no other tempered connected
component can occur.
\end{remark}

\subsection{Adjacent-length formulation}

We now explain Conjecture~\ref{conj:block}
in terms of the lengths of the discrete-series factors.

Let $\pi$ be tempered and group its discrete-series factors according to
supercuspidal lines. Fix a unitary supercuspidal representation $\rho$.
The part of $\pi$ supported on the $\rho$-line can be written uniquely,
up to permutation, as
\[
 \pi_\rho
 =
 \delta_1(\rho)^{\times c_1(\rho)}
 \times
 \delta_2(\rho)^{\times c_2(\rho)}
 \times
 \delta_3(\rho)^{\times c_3(\rho)}
 \times\cdots ,\]
where $c_j(\rho)\geq0$ and only finitely many $c_j(\rho)$ are nonzero.

We record these factors by their lengths and associate to $\pi_\rho$ the
finite multiset
\[
 \mathcal L_\rho(\pi)
 =
 \{\underbrace{1,\ldots,1}_{c_1(\rho)},
 \underbrace{2,\ldots,2}_{c_2(\rho)},
 \underbrace{3,\ldots,3}_{c_3(\rho)},\ldots\}.
\]
For example,
$\delta_3(\rho)\times\delta_2(\rho)\times\rho\times\rho
$
corresponds to the length multiset $ \{3,2,1,1\}.$

The conjecture says that, on a distinguished supercuspidal line, these
lengths must be assembled in only two ways:
\begin{equation}\label{eq:allowed-length-pieces}
	\{1\}
	\qquad\text{or}\qquad
	\{m,m+1\},\qquad m\geq1.
\end{equation}
The singleton $\{1\}$ corresponds to the distinguished supercuspidal factor
$\rho$, while the adjacent pair $\{m,m+1\}$ corresponds to the elementary
block
\[
 \cB_m(\rho)
 =
 \delta_{m+1}(\rho)\times\delta_m(\rho).
\]

%

Thus, for instance,
$ \{1,1,2,3,4,5\}
 =
 \{1\}\sqcup\{1\}\sqcup\{2,3\}\sqcup\{4,5\}$
is allowed. It corresponds to
$ \rho\times\rho
 \times\cB_2(\rho)\times\cB_4(\rho).
$
On the other hand,
$\{2,4\},
 \{1,3\},
\{2,2,3\}
$
are not allowed, since their occurrences cannot be exhausted by the pieces
listed in \eqref{eq:allowed-length-pieces}.

The adjacent-length description should be understood for each individual occurrence, and not only in terms of the lengths which appear.
For example, the multiset
$ \{2,2,3,3\}$
is allowed, since it is
$ \{2,3\}\sqcup\{2,3\},
$
whereas
 $\{2,2,3\}$
is not. 

\subsection{Parameter formulation}

We use the local Langlands correspondence for inner forms of general linear groups, normalized compatibly with the local Jacquet--Langlands correspondence; see  \cite{ABPS2016,MSSS2026}.
Let $S_a=\Sym^{a-1}(\C^2)$
denote the $a$-dimensional irreducible representation of
$\SL_2(\C)$. For $\rho\in\cD_r$, write its Langlands parameter as
\begin{equation}\label{eq:rho-param}
 \ph_\rho=\tau_\rho\otimes S_2,
\end{equation}
where $\tau_\rho$ is an irreducible $r$-dimensional representation of
$W_F$. The parameter of $\delta_m(\rho)$ is
\begin{equation}\label{eq:delta-param}
 \ph_{\delta_m(\rho)}
 =
 \tau_\rho\otimes S_{2m}.
\end{equation}
Consequently,
\begin{align*}
 \ph_{\cB_m(\rho)}
 &=
 \tau_\rho\otimes
 \left(S_{2m+2}\oplus S_{2m}\right) \notag\\
 &=
 \tau_\rho\otimes
 \left(S_{2m+1}\otimes S_2\right),
 \qquad m\geq1, \label{eq:block-param}
\end{align*}
where the last identity is the Clebsch--Gordan formula.

\begin{conjecture}[Tempered distinction: parameter form]\label{conj:param}
Let $\pi$ be an irreducible tempered representation of $G_n$. Then $\pi$
is $H_n$-distinguished if and only if its parameter is a direct sum of
pieces of the following two types:
\begin{align*}
 &\tau\otimes S_2,\\
 &\tau\otimes\left(S_{2m}\oplus S_{2m+2}\right),
 \qquad m\geq1,
\end{align*}
where in each case $\tau\otimes S_2$ is the parameter of a distinguished
supercuspidal representation of some $G_r$.
\end{conjecture}

Thus, for each fixed irreducible Weil representation $\tau$, the even
$\SL_2$-Jordan blocks occurring in a distinguished tempered parameter can be
partitioned into
$S_2
 \qquad\text{and}\qquad
 S_{2m}\oplus S_{2m+2}.
$
Equivalently, the corresponding even integers can be partitioned into
$
 \{2\}
 \qquad\text{and}\qquad
 \{2m,2m+2\}.
$

\begin{proposition}
Conjectures~\ref{conj:block} and \ref{conj:param} are equivalent.
\end{proposition}

\begin{proof}
This follows from \eqref{eq:rho-param}, \eqref{eq:delta-param}, additivity of
Langlands parameters under tempered parabolic induction, and
\[
 S_{2m+1}\otimes S_2
 \simeq
 S_{2m+2}\oplus S_{2m}.
\]
\end{proof}

\begin{remark}[Parameter interpretation]\label{rem:adj}
For a distinguished supercuspidal $\rho$, the reducibility step is
$s(\rho)=2$. Thus the cuspidal support of $\delta_m(\rho)$, in
$\nu$-exponents, is
$
 \{1-m,3-m,\ldots,m-3,m-1\}.
$
The two supports
$
 \operatorname{Supp}(\delta_{m+1}(\rho))
 \quad\text{and}\quad
 \operatorname{Supp}(\delta_m(\rho))
$
are the two parity classes in the consecutive interval
$
 \{-m,-m+1,\ldots,m-1,m\}.
$
This is the local combinatorial reason for the occurrence of adjacent
lengths.
\end{remark}

\section{Existence of the elementary blocks $\cB_m(\rho)$}
\label{sec:Bm-general}

In this section we prove the existence direction of
Conjecture~\ref{conj:block}. Namely, every elementary block
\[
   \cB_m(\rho)
   =
   \delta_{m+1}(\rho)\times\delta_m(\rho)
\]
attached to a distinguished supercuspidal representation $\rho$ has a
symplectic model. The proof is local and uses the closed orbit in the
geometric-lemma filtration for a maximal parabolic subgroup.

Throughout this section let
$\rho\in\cD_d$
be a unitary $H_d$-distinguished supercuspidal representation of $G_d$.
By Theorem~\ref{thm:ds},
\begin{equation}\label{eq:rho-s-two-section5}
   s(\rho)=2,
   \qquad
   \nu_\rho=\nu^2.
\end{equation}
Fix $m\geq1$ and put
$
   k=md,
   \qquad
   n=(2m+1)d.
$

Then
$ n-k=(m+1)d,
   \qquad
   n-2k=d$
   and 
the block $\cB_m(\rho)$ is a representation of $G_n$.

\subsection{The relevant maximal parabolic and its closed orbit}

Let $
   P=P_{k,n-k}=MU
   =
   P_{md,(m+1)d},
$ be the maximal parbaolic subgroup of $G_n$
with Levi subgroup
$
   M\simeq G_k\times G_{n-k}
   =
   G_{md}\times G_{(m+1)d}.
$
We consider
\begin{equation}\label{eq:Bm-induced}
   \Pi_m
   =
   \Ind_{P}^{G_n}
   \left(
      \delta_m(\rho)\otimes\delta_{m+1}(\rho)
   \right).
\end{equation}
The two centered segments
$
   \left[-\frac{m-1}{2},\frac{m-1}{2}\right]_\rho
   \quad\text{and}\quad
   \left[-\frac m2,\frac m2\right]_\rho
$
are nested. Hence they are not linked, and the standard irreducibility
criterion for products of essentially square-integrable representations gives
\begin{equation}\label{eq:Bm-irred}
   \Pi_m
   \simeq
   \delta_m(\rho)\times\delta_{m+1}(\rho)
   \simeq
   \delta_{m+1}(\rho)\times\delta_m(\rho)
   =
   \cB_m(\rho).
\end{equation}
See \cite{Tadic1990,LapidMinguez2014}.

For this parabolic the orbit geometry is as follows. The quotient
$ P_{k,n-k}\backslash G_n
$
may be identified with the Grassmannian of $k$-dimensional right
$D$-subspaces of the underlying $n$-dimensional Hermitian space. The
$H_n$-orbits are indexed by
\[
   O_r
   =
   \left\{
      W:\dim_D W=k,\
      \dim_D\operatorname{Rad}(J|_W)=r
   \right\},
   \qquad
   0\leq r\leq k.
\]
This is the maximal-parabolic orbit decomposition used in
\cite[\S3.4]{SharmaVerma2026}; it is a special case of the general symmetric
space framework in \cite{Offen2017}.

The orbit $O_k$ consists of totally isotropic $k$-planes. Since the
condition
$
   J|_W=0
$
is closed on the Grassmannian, $O_k$ is the closed orbit. Its stabilizer is
a parabolic subgroup
$ H_{k,k}=M_{k,k}U_{k,k}
$
of $H_n$ whose Levi factor is
\[
   M_{k,k}
   \simeq
   \Delta(G_k\times G_k)\times H_{n-2k}
   =
   \Delta(G_{md}\times G_{md})\times H_d.
\]
Thus the closed orbit leaves a single fixed symplectic factor of rank $d$,
while the two factors of rank $md$ are exchanged.

\subsection{Modulus characters and the closed-orbit Hom-space}

We fix the normalization in the geometric lemma. Let
$\delta_P$ denote the modulus character of $P=P_{k,n-k}$. For
 $ \operatorname{diag}(g_1,g_2)\in G_k\times G_{n-k}$
we have
\[
   \delta_P
   \bigl(\operatorname{diag}(g_1,g_2)\bigr)
   =
   \nu(g_1)^{2(n-k)}
   \nu(g_2)^{-2k}.
\]

Let $\delta_{k,r}$ be the modulus character of the stabilizer
$H_{k,r}$. If
\[
   m_r
   =
   \operatorname{diag}
   (g,h_1,h_2,g^\ast)
   \in M_{k,r},
\]
with
$ g\in G_r,\qquad
   h_1\in H_{k-r},\qquad
   h_2\in H_{n-k-r},
$
the calculation in \cite[Proposition~3.1]{SharmaVerma2026}, compared with
the general orbit calculation in \cite{Offen2017}, gives
\[
   \delta_{k,r}^{-1}(m_r)
   \delta_P^{1/2}(m_r)
   =
   \nu(g)^{-(n-2r+1)}.\]
In the closed-orbit case $r=k$ this becomes
\begin{equation}\label{eq:stabilizer-mod-closed}
   \delta_{k,k}^{-1}(m_k)
   \delta_P^{1/2}(m_k)
   =
   \nu(g)^{-(n-2k+1)}
   =
   \nu(g)^{-(d+1)}.
\end{equation}

After passing to normalized Jacquet functors, only a single power of
$\nu$ remains. For the closed orbit the first inducing factor
$\delta_m(\rho)$ is not Jacquet-refined, whereas the second inducing factor
$\delta_{m+1}(\rho)$ is refined with respect to the partition
$(n-2k,k)=(d,md).$

Let $N$ denote the unipotent radical of the corresponding parabolic
$P_{d,md}$ in $G_{(m+1)d}$. On the $G_k$-variable $g$ embedded in the
second block as $g^\ast$, the square root of the Jacquet modulus contributes
$
   \delta_N^{1/2}=\nu(g)^d.
$
Combining this with \eqref{eq:stabilizer-mod-closed} yields
\[
   \nu(g)^{-(d+1)}\nu(g)^d
   =
   \nu(g)^{-1}.
\]
Moving this character to the target of the Hom-space gives $\nu(g)$,
which is the single $\nu$-twist in the good-orbit criterion.

Corollary~4.1 of \cite{SharmaVerma2026}, based on
\cite{BernsteinZelevinsky1977,Offen2017}, gives
\begin{equation}\label{eq:closed-Hom-general}
\begin{split}
   &\Hom_{H_n}
   \left(
      \text{closed-orbit quotient of }\Pi_m|_{H_n},
      \C
   \right)\\
   &\hspace{1cm}\simeq
   \Hom_{M_{k,k}}
   \left(
      \delta_m(\rho)
      \otimes
      r_{(d,md),((m+1)d)}
      \bigl(\delta_{m+1}(\rho)\bigr),
      \nu
   \right).
\end{split}
\end{equation}
On
$ M_{k,k}
   \simeq
   \Delta(G_k\times G_k)\times H_d,
$
the character $\nu$ in \eqref{eq:closed-Hom-general} is the character
\[
   (g,h)\longmapsto \nu(g).
\]

We use the following specialization of Proposition~4.2 of
\cite{SharmaVerma2026}, with the orbit input from \cite{Offen2017}.

\begin{lemma}[Closed-orbit pairing criterion]\label{lem:closed-general}
Suppose that
$ r_{(d,k),(d+k)}(\lambda_2)
$
contains an irreducible constituent
$
   \eta_0\otimes\eta_2,
   \qquad
   \eta_0\in\Irr(G_d),\quad
   \eta_2\in\Irr(G_k).
$
Then the constituent
$
   \lambda_1\otimes\eta_0\otimes\eta_2
$
contributes to the closed-orbit Hom-space in
\eqref{eq:closed-Hom-general} provided
\begin{equation}\label{eq:closed-two-conditions}
   \lambda_1\simeq\nu\eta_2
   \qquad\text{and}\qquad
   \Hom_{H_d}(\eta_0,\C)\neq0.
\end{equation}
\end{lemma}

\begin{proof}
For the closed orbit the involution on the three refined Levi blocks exchanges
the two $G_k$-blocks and fixes the middle $G_d$-block. Proposition~4.2 of \cite{SharmaVerma2026}, with the orbit description of
\cite{Offen2017}, gives the two conditions in
\eqref{eq:closed-two-conditions}.
\end{proof}

\subsection{The Jacquet module of $\delta_{m+1}(\rho)$}

The required Jacquet module is computed as follows.

Recall that
 $ \delta_{m+1}(\rho)
   =
   Q\left(
      \left[-\frac m2,\frac m2\right]_\rho
   \right).
$
The Jacquet module with respect to the partition
$
   (d,md)
$
cuts this length-$(m+1)$ segment into a terminal piece of length one and a
remaining piece of length $m$. Thus
\[
   \left[-\frac m2,\frac m2\right]_\rho
   =
   \left[\frac m2,\frac m2\right]_\rho
   \sqcup
   \left[-\frac m2,\frac m2-1\right]_\rho
\]
in the ordered sense of the Jacquet-module formula. Consequently,
\begin{equation}\label{eq:jacquet-first-general}
\begin{split}
   r_{(d,md),((m+1)d)}
   \bigl(\delta_{m+1}(\rho)\bigr)
   &=
   Q\left(
      \left[\frac m2,\frac m2\right]_\rho
   \right)\\
   &\qquad\otimes
   Q\left(
      \left[-\frac m2,\frac m2-1\right]_\rho
   \right).
\end{split}
\end{equation}

The two factors can be written more explicitly. By \eqref{eq:rho-s-two-section5},
 $ \nu_\rho=\nu^2,$
hence
\begin{equation}\label{eq:jacquet-fixed-general}
   Q\left(
      \left[\frac m2,\frac m2\right]_\rho
   \right)
   =
   \nu_\rho^{m/2}\rho
   =
   \nu^m\rho.
\end{equation}
For the second factor,
\[
   \left[-\frac m2,\frac m2-1\right]_\rho
   =
   \nu_\rho^{-1/2}
   \left[
      -\frac{m-1}{2},
      \frac{m-1}{2}
   \right]_\rho.
\]
Therefore
\begin{equation}\label{eq:jacquet-paired-general}
\begin{split}
   Q\left(
      \left[-\frac m2,\frac m2-1\right]_\rho
   \right)
   &=
   \nu_\rho^{-1/2}
   Q\left(
      \left[
         -\frac{m-1}{2},
         \frac{m-1}{2}
      \right]_\rho
   \right)\\
   &=
   \nu^{-1}\delta_m(\rho).
\end{split}
\end{equation}
Combining \eqref{eq:jacquet-first-general},
\eqref{eq:jacquet-fixed-general}, and
\eqref{eq:jacquet-paired-general}, we obtain
\begin{equation}\label{eq:jacquet-Bm-master}
 r_{(d,md),((m+1)d)}
   \bigl(\delta_{m+1}(\rho)\bigr)
   =
   \nu^m\rho
   \otimes
   \nu^{-1}\delta_m(\rho).
\end{equation}

At the level of $\nu$-exponents this identity is especially transparent.
The support of $\delta_{m+1}(\rho)$ is
  $ \{-m,-m+2,\ldots,m-2,m\}.$

Removing the terminal exponent $m$ leaves
\[
   \{-m,-m+2,\ldots,m-2\}
   =
   -1+
   \{-(m-1),-(m-3),\ldots,m-1\},
\]
which is the support of
$\nu^{-1}\delta_m(\rho)$.

\subsection{Existence of the symplectic period}

The preceding calculation gives the general existence result.

\begin{theorem}[Existence of the elementary blocks]
\label{thm:Bm-general}
Let $\rho\in\cD_d$ be a unitary distinguished supercuspidal representation of
$G_d$, and let $m\geq1$. Then
\[
   \cB_m(\rho)
   =
   \delta_{m+1}(\rho)\times\delta_m(\rho)
\]
admits a symplectic model. Equivalently,
\begin{equation}\label{eq:Bm-distinguished-general}
   \Hom_{H_{(2m+1)d}}
   \left(
      \cB_m(\rho),\C
   \right)
   \neq0.
\end{equation}
\end{theorem}

\begin{proof}
Consider the irreducible induced representation $\Pi_m$ in
\eqref{eq:Bm-induced}. By \eqref{eq:Bm-irred}, it is isomorphic to
$\cB_m(\rho)$.
We use the closed orbit $O_k$ in
\[
   P_{md,(m+1)d}\backslash G_{(2m+1)d}/H_{(2m+1)d}.
\]
By \eqref{eq:jacquet-Bm-master}, the relevant Jacquet-module constituent of
the second inducing factor is
$ \nu^m\rho
   \otimes
   \nu^{-1}\delta_m(\rho).
$
Thus the three refined factors attached to the closed orbit are
$
   \delta_m(\rho)
   \otimes
   \nu^m\rho
   \otimes
   \nu^{-1}\delta_m(\rho).
$
The two exchanged $G_{md}$-factors satisfy
$  \delta_m(\rho)
   =
   \nu\bigl(\nu^{-1}\delta_m(\rho)\bigr).$
Moreover,
$
   \Hom_{H_d}(\nu^m\rho,\C)
   =
   \Hom_{H_d}(\rho,\C)
   \neq0,
$
because $\nu$ is trivial on $H_d$. Therefore
Lemma~\ref{lem:closed-general} implies that the closed-orbit quotient has a
nonzero $H_n$-invariant linear form:
\[
   \Hom_{H_n}
   \left(
      \text{closed-orbit quotient of }\Pi_m|_{H_n},
      \C
   \right)
   \neq0.
\]

Since $O_k$ is closed, its term is the final quotient in the geometric-lemma
filtration. Hence there is a canonical $H_n$-equivariant surjection
\[
   \Pi_m|_{H_n}
   \twoheadrightarrow
   \text{closed-orbit quotient}.
\]
Pulling back a nonzero invariant form through this surjection gives a
nonzero element of
$\Hom_{H_n}(\Pi_m,\C).
$
Finally $\Pi_m\simeq\cB_m(\rho)$, so
\eqref{eq:Bm-distinguished-general} follows.
\end{proof}

\begin{remark}[The orbit involution]
The refined Levi for the closed orbit has three blocks of ranks
$md,\quad d,\quad md.
$
The corresponding involution is
$
   1\longleftrightarrow3,
   \qquad
   2\longmapsto2.
$
Thus the proof of Theorem~\ref{thm:Bm-general} is represented by the
$
   \delta_m(\rho)
   \quad\longleftrightarrow\quad
   \nu^{-1}\delta_m(\rho),
   \qquad
   \nu^m\rho\ \text{fixed}.
$
The $\nu$-pairing of the first two factors comes from the modulus character
of the closed orbit, whereas the fixed factor survives because the
supercuspidal representation $\rho$ is itself distinguished.
\end{remark}

\begin{remark}[Comparison with the non-distinguished supercuspidal line]
Section~5.2 of \cite{SharmaVerma2026}, using the symplectic-multiset
argument of \cite{MitraOffenSayag2017}, treats a supercuspidal line generated
by a representation which does not have a symplectic model. There a good
decomposition has no fixed points: every refined segment must occur in a
$\nu$-paired two-cycle. Theorem~\ref{thm:Bm-general} shows what changes on a distinguished
supercuspidal line. A single fixed piece,
namely $\nu^m\rho$, is now allowed, and the remainder of
$\delta_{m+1}(\rho)$ pairs with the external factor
$\delta_m(\rho)$.
\end{remark}

\begin{corollary}[The ``if'' direction of the tempered conjecture]
\label{cor:if-tempered}
Let
\[
   \pi
   =
   \cB_{m_1}(\rho_1)
   \times\cdots\times
   \cB_{m_t}(\rho_t),
\]
where $\rho_i\in\cD$ and $m_i\geq0$. Then $\pi$ has a symplectic model.
\end{corollary}

\begin{proof}
If $m_i=0$, then $\cB_0(\rho_i)=\rho_i$ is distinguished by definition.
If $m_i\geq1$, then $\cB_{m_i}(\rho_i)$ is distinguished by
Theorem~\ref{thm:Bm-general}. The hereditary property of the symplectic
model, \cite[Theorem~4.1]{SharmaVerma2026}, now gives the result; see also
\cite[Proposition~7.1]{Offen2017} and
\cite[Theorem~5.4]{MatringeOffenYang2024}.
\end{proof}

\begin{remark}
The converse, proved in Section~\ref{sec:converse}, shows that every
irreducible tempered distinguished representation is a product of the
blocks above.
\end{remark}

We record the rank-five specialization of the general classification. Since a
distinguished supercuspidal representation of $G_r$ has $r$ odd, only
$r=1,3,5$ can occur when the total rank is five.

Let $\cD_1$ be the set of unitary characters of $D^\times$; equivalently,
the unitary characters factoring through the reduced norm. Every such
character is $H_1=D^1$-distinguished.

%
%

\begin{theorem}[Tempered distinction for $G_5$]
\label{thm:rank5}
Let $\pi$ be an irreducible tempered representation of $G_5$. Then
\[
 \Hom_{H_5}(\pi,\C)\neq0
\]
if and only if, up to permutation of the inducing factors, $\pi$ belongs to
one of the following families:
\begin{enumerate}[label=\textup{(\arabic*)}]
\item $\rho_5$, where $\rho_5\in\cD_5$;
\item $\rho_3\times\chi_1\times\chi_2$, where
 $\rho_3\in\cD_3$ and $\chi_1,\chi_2\in\cD_1$;
\item $\chi_1\times\chi_2\times\chi_3\times\chi_4\times\chi_5$, where
 $\chi_i\in\cD_1$;
\item $\cB_1(\chi)\times\chi_1\times\chi_2$, where
 $\chi,\chi_1,\chi_2\in\cD_1$;
\item $\cB_2(\chi)=\delta_3(\chi)\times\delta_2(\chi)$, where
 $\chi\in\cD_1$.
\end{enumerate}
\end{theorem}

\begin{remark}
Theorem~\ref{thm:rank5} follows from the general classification, Theorem~\ref{thm:main-classification}, proved in
Section~\ref{sec:converse}. Indeed, an elementary block
$\cB_m(\rho)$ has rank $(2m+1)r$ when $\rho\in\cD_r$. In total rank five,
a non-supercuspidal block therefore has $r=1$ and is either
$\cB_1(\chi)$ or $\cB_2(\chi)$; the remaining rank is filled by
distinguished supercuspidal singletons.
\end{remark}

\section{The converse theorem: rank reduction}
\label{sec:converse}

Let
\[
   \pi=\delta_L(\rho)\times\sigma
\]
be tempered and distinguished, with $L$ maximal on the $\rho$-line. The
Mackey filtration for $P_{Ld,n-Ld}$ first shows that $\rho$ is
distinguished. It then determines the contributing orbit and, for $L>1$,
forces $\delta_{L-1}(\rho)$ to occur as a factor of $\sigma$. The
remaining factor is again distinguished. This is the rank reduction used
for the induction.

\subsection{The maximal-parabolic Mackey filtration}

Write a tempered representation as
\begin{equation}\label{eq:general-temp-converse}
   \pi=\delta_L(\rho)\times\sigma,
\end{equation}
where $\rho$ is a unitary supercuspidal representation of $G_d$ and
$\delta_L(\rho)$ is chosen with $L$ maximal among the lengths occurring on
the $\rho$-line. Since normalized induction of unitary discrete-series
representations of an inner form of a general linear group is irreducible and
independent of the order of the factors, we may place this factor first; see
\cite{Tadic1990,LapidMinguez2016}.

Put
$
   k=Ld,
   \qquad
   N=n-k,
   \qquad
   P=P_{k,N}.
$
The $H_n$-orbits in the maximal-parabolic Mackey filtration are indexed by
integers
\[
   0\le r\le\min\{k,N\}.
\]
Let $J_r$ denote the corresponding graded piece. Corollary~4.1 of \cite{SharmaVerma2026}, based on
\cite{BernsteinZelevinsky1977,Offen2017}, gives
\begin{equation}\label{eq:maximal-mackey-hom}
\begin{split}
 &\Hom_{H_n}(J_r,\C)\\
 &\qquad\simeq
 \Hom_{M_{k,r}}
 \left(
   r_{(r,k-r),(k)}\bigl(\delta_L(\rho)\bigr)
   \otimes
   r_{(N-r,r),(N)}(\sigma),
   \nu
 \right),
\end{split}
\end{equation}
where
$
   M_{k,r}
   \simeq
   \Delta(G_r\times G_r)
   \times H_{k-r}
   \times H_{N-r}.
$
If \eqref{eq:maximal-mackey-hom} is nonzero, Proposition~4.2 of \cite{SharmaVerma2026}, with the orbit description of
\cite{Offen2017}, implies that, after passing to an irreducible
subquotient
\[
   S\otimes T
   \quad\text{of}\quad
   r_{(N-r,r),(N)}(\sigma),
\]
the two $G_r$-pieces are related by a single $\nu$-twist, while the pieces
on the two symplectic factors are distinguished.

The Jacquet module of $\delta_L(\rho)$ is zero unless
$
   r=td
   \qquad(0\le t\le L).
$
For $0<t<L$, Lemma~\ref{lem:two-block-jacquet} gives
\begin{equation}\label{eq:maximal-jacquet-deltaL}
\begin{split}
 r_{(td,(L-t)d),(Ld)}\bigl(\delta_L(\rho)\bigr)
   &=A_t\otimes B_t,\\
 A_t&=\nu^{s(\rho)(L-t)/2}\delta_t(\rho),\\
 B_t&=\nu^{-s(\rho)t/2}\delta_{L-t}(\rho).
\end{split}
\end{equation}
With our ordering convention, nonvanishing of the $r$-th Mackey quotient
therefore forces
\begin{equation}\label{eq:maximal-three-conditions}
   A_t\simeq\nu T,
   \qquad
   \Hom_{H_{(L-t)d}}(B_t,\C)\neq0,
   \qquad
   \Hom_{H_{N-td}}(S,\C)\neq0.
\end{equation}
For $t=0$ the whole factor $\delta_L(\rho)$ is fixed by $H_{Ld}$; if the
full-pairing orbit $t=L$ is available, the whole $\delta_L(\rho)$ is paired
with the terminal $G_{Ld}$-piece of a Jacquet module of $\sigma$.

The following Jacquet-module formula and the terminal-support observation
will be used in the rank-reduction argument.

\begin{lemma}[Jacquet module of the maximal factor]
\label{lem:maximal-factor-jacquet}
Let $\rho\in\cD_d$ and $L\ge2$. For $0<t<L$,
\[
 r_{(td,(L-t)d),(Ld)}\bigl(\delta_L(\rho)\bigr)
 =
 \nu^{L-t}\delta_t(\rho)
 \otimes
 \nu^{-t}\delta_{L-t}(\rho).
\]
\end{lemma}

\begin{proof}
Since $\rho$ is distinguished, $s(\rho)=2$. The assertion is therefore the
specialization of \eqref{eq:maximal-jacquet-deltaL} to $s(\rho)=2$.
\end{proof}

\begin{remark}[Terminal pieces of centered segments]
\label{rem:terminal-piece-shape}
Let
$   \sigma
   =
   \delta_{K_1}(\rho_1)\times\cdots\times\delta_{K_u}(\rho_u)
$
be tempered, and consider a Jacquet-module constituent
$ S\otimes T
   \quad\text{of}\quad
   r_{(N-r,r),(N)}(\sigma).
$
Suppose that $T$ lies on the cuspidal line of a unitary supercuspidal
representation $\rho$, and write $s=s(\rho)$.

The possible contribution to $T$ can be read from the endpoint of the
centered segment.
The representation $\delta_K(\rho)$ corresponds to the centered segment
\[
   \left[-\frac{K-1}{2},\frac{K-1}{2}\right]_\rho .
\]
When a Jacquet module separates off a piece into the second Levi factor,
that piece is taken from the lower end of this segment. This is the usual
segment description of Jacquet modules; see, for example,
\cite{Tadic1990} and \cite[\S2.4]{SharmaVerma2026}. Thus, if
$\delta_K(\rho)$ contributes nontrivially to $T$, the contribution starts
at the left endpoint of the segment. Since $\nu_\rho=\nu^{s}$, this
endpoint is
\[
   -\frac{s(K-1)}{2}
\]
when written in ordinary $\nu$-exponents.

Thus the initial exponent of a terminal piece determines the length $K$ of
the discrete-series factor from which it comes. We use this below to compare
terminal Jacquet pieces with the contribution from $\delta_L(\rho)$.
\end{remark}

\begin{remark}[Endpoint rigidity]
	\label{rem:adjacent-terminal-rigidity}
	In the situation of Remark~\ref{rem:terminal-piece-shape}, assume that
	$s(\rho)=2$, every factor on the $\rho$-line has length at most $L$, and
	\[
	T\simeq\delta_{L-1}(\rho).
	\]
	Then $T$ must come from one whole original factor
	$\delta_{L-1}(\rho)$ of $\sigma$; it cannot be assembled from proper
	terminal pieces of several factors.
	
	Indeed, the support of $T$ runs from $2-L$ to $L-2$. If a centered factor
	$\delta_K(\rho)$ supplies the top exponent $L-2$, then its terminal piece
	starts at $1-K$. Since this starting point must still lie in the support of
	$T$, one has $K\le L-1$; on the other hand, reaching the exponent $L-2$
	forces $K\ge L-1$. Hence $K=L-1$. A proper terminal piece of
	$\delta_{L-1}(\rho)$ stops before $L-2$, so this factor must contribute its
	whole segment. Its rank is already $(L-1)d$, the full rank of $T$, and
	there is therefore no room for any additional contribution.
\end{remark}
\subsection{Distinguished cuspidal support}

\begin{proposition}[Distinguished cuspidal support]
	\label{prop:distinguished-cuspidal-support}
	Every supercuspidal line occurring in the tempered support of an
	$H_n$-distinguished tempered representation contains a distinguished
	supercuspidal representation.
\end{proposition}

\begin{proof}
	Fix a unitary supercuspidal representation $\rho$ whose line occurs in the
	tempered support, and choose $\delta_L(\rho)$ of maximal length on that
	line. Reorder the tempered induction as in
	\eqref{eq:general-temp-converse}. Suppose that $\rho$ is not distinguished.
	Then, by \cite[Remark~1.1]{SharmaVerma2026}, using the Jacquet--Langlands
	reducibility calculation of \cite[Theorem~B.2.b]{DeligneKazhdanVigneras1984},
	and the discrete-series
	distinction theorem,
$
	s(\rho)=1.
$
	Since $\pi$ is distinguished, Lemma~\ref{lem:dist-implies-good} gives a good
	orbit in the $P_{Ld,N}$-Mackey filtration. Write $r=td$.
	
	If $t=0$, the open-orbit condition requires
	\[
	\Hom_{H_{Ld}}(\delta_L(\rho),\C)\neq0.
	\]
	For $L=1$ this contradicts the assumption on $\rho$; for $L>1$ it
	contradicts Theorem~\ref{thm:ds}.
	
	Assume $0<t<L$. By \eqref{eq:maximal-three-conditions},
$
	B_t=\nu^{-t/2}\delta_{L-t}(\rho)
$
	is distinguished. Since $\nu$ is trivial on every $H_j$, this is equivalent
	to distinction of $\delta_{L-t}(\rho)$. If $L-t>1$, this contradicts the
	discrete-series distinction theorem; if $L-t=1$, it says that $\rho$ is
	distinguished. Thus no interior orbit can be good.
	
	It remains to exclude the full-pairing orbit $t=L$. If this orbit is
	available, Proposition~4.2 of \cite{SharmaVerma2026}, with the orbit
	analysis of \cite{Offen2017}, forces a terminal
	constituent $T$ of a Jacquet module of $\sigma$ with
	\[
	\delta_L(\rho)\simeq\nu T,
	\qquad\text{hence}\qquad
	T\simeq\nu^{-1}\delta_L(\rho).
	\]
	We now exclude this possibility directly. Since $s(\rho)=1$, the ordinary
	$\nu$-cuspidal support of $\delta_L(\rho)$ is
	\[
	\left\{
	\nu^{-(L-1)/2}\rho,\nu^{-(L-3)/2}\rho,\ldots,
	\nu^{(L-1)/2}\rho
	\right\}.
	\]
	Consequently
	\[
	\operatorname{cusp}\bigl(\nu^{-1}\delta_L(\rho)\bigr)
	=
	\left\{
	\nu^{-(L+1)/2}\rho,\nu^{-(L-1)/2}\rho,\ldots,
	\nu^{(L-3)/2}\rho
	\right\}.
	\]
	In particular, $\nu^{-(L+1)/2}\rho$ occurs in the cuspidal support of $T$,
	and therefore in the cuspidal support of $\sigma$. Because $\sigma$ is
	tempered, every discrete-series factor on the $\rho$-line is centered, say
	$\delta_K(\rho)$. Hence a factor containing
	$\nu^{-(L+1)/2}\rho$ also contains $\nu^{(L+1)/2}\rho$. Since the extreme
	ordinary exponents of $\delta_K(\rho)$ are
	\[
	-\frac{K-1}{2}\qquad\text{and}\qquad \frac{K-1}{2},
	\]
	we must have
	\[
	\frac{K-1}{2}\geq \frac{L+1}{2},
	\qquad\text{so}\qquad K\geq L+2.
	\]
	This contradicts the maximality of $L$ on the $\rho$-line. Thus the
	full-pairing orbit is not good either.
	
	All possible orbits have been excluded, contradicting
	Lemma~\ref{lem:dist-implies-good}. Hence $\rho$ is distinguished.
\end{proof}

\subsection{The rank-reduction argument}

We now turn to the rank reduction. Suppose
\[
   \pi=\delta_L(\rho)\times\sigma
\]
is irreducible tempered and $H_n$-distinguished, where
$\rho\in\cD_d$, $L\ge2$, and $L$ is maximal among the lengths of the
factors on the $\rho$-line. Put
\[
   n_0=n-(2L-1)d.
\]
The reduction has the form
\begin{equation}\label{eq:rank-reduction-mechanism}
\begin{aligned}
 \pi=\delta_L(\rho)\times\sigma,\quad \Hom_{H_n}(\pi,\C)\neq0
 &\Longrightarrow r=(L-1)d,\\
 &\Longrightarrow T\simeq\delta_{L-1}(\rho),\\
 &\Longrightarrow \delta_{L-1}(\rho)\text{ is a whole factor of }\sigma,\\
 &\Longrightarrow \sigma\simeq\delta_{L-1}(\rho)\times\pi_0,\\
 &\Longrightarrow \Hom_{H_{n_0}}(\pi_0,\C)\neq0.
\end{aligned}
\end{equation}
Here $T$ denotes the terminal $G_{(L-1)d}$-factor in the Jacquet
constituent of $\sigma$ selected by the good orbit, and the complementary
Jacquet factor will be denoted by $S$.

To prove \eqref{eq:rank-reduction-mechanism}, use the maximal-parabolic
Mackey filtration for $P_{Ld,n-Ld}$. The open orbit
is excluded by the discrete-series distinction theorem. The full-pairing
closed orbit $r=Ld$ is excluded by Lemma~\ref{lem:closed-orbit-maximality}:
the diagonal pairing would force
\[
   T\simeq\nu^{-1}\delta_L(\rho),
\]
so $\nu^{-L}\rho$ would occur in the cuspidal support of $\sigma$; since
$\sigma$ is tempered, the relevant $\rho$-line factor is centered and would
also contain $\nu^{L}\rho$, forcing a factor of length at least $L+1$ and
contradicting maximality. For an interior orbit $r=td$, the fixed
symplectic factor forces $L-t=1$, hence $r=(L-1)d$.

For that unique orbit, the diagonal condition forces
$T\simeq\delta_{L-1}(\rho)$. The endpoint-rigidity observation in
Remark~\ref{rem:adjacent-terminal-rigidity} then shows that this $T$ cannot
be assembled from proper Jacquet pieces: one original factor of $\sigma$
is $\delta_{L-1}(\rho)$ and contributes wholly. Thus, after
reordering,
\[
   \sigma\simeq\delta_{L-1}(\rho)\times\pi_0,
   \qquad S\simeq\pi_0.
\]
Finally the remaining fixed symplectic factor in the orbit stabilizer is
$H_{n_0}$, so the good-orbit condition gives
$\Hom_{H_{n_0}}(\pi_0,\C)\neq0$.

The elementary-block reduction is the case $L=m+1$. We isolate first the
closed-orbit exclusion used in the argument.

\begin{lemma}[Closed-orbit exclusion on a distinguished line]
\label{lem:closed-orbit-maximality}
Let $\rho\in\cD_d$, let $L\geq1$, and let $\sigma$ be tempered. Assume that
$L$ is maximal among the lengths of the discrete-series factors occurring on
the $\rho$-cuspidal line in $\delta_L(\rho)\times\sigma$. Then no irreducible
constituent $T$ of a normalized Jacquet module of $\sigma$ can satisfy
\[
   T\simeq \nu^{-1}\delta_L(\rho).
\]
Consequently, in the maximal-parabolic Mackey filtration for
$\delta_L(\rho)\times\sigma$, the full-pairing (closed) orbit $r=Ld$ is not
good.
\end{lemma}

\begin{proof}
Since $\rho$ is distinguished, $s(\rho)=2$. Hence the ordinary
$\nu$-cuspidal support of $\delta_L(\rho)$ is
\[
   \{\nu^{1-L}\rho,\nu^{3-L}\rho,\ldots,\nu^{L-1}\rho\},
\]
and therefore
\[
   \operatorname{cusp}\bigl(\nu^{-1}\delta_L(\rho)\bigr)
   =
   \{\nu^{-L}\rho,\nu^{2-L}\rho,\ldots,\nu^{L-2}\rho\}.
\]
In particular, if a Jacquet constituent $T$ of $\sigma$ were isomorphic to
$\nu^{-1}\delta_L(\rho)$, then $\nu^{-L}\rho$ would occur in the cuspidal
support of $T$, and hence in the cuspidal support of $\sigma$. Normalized
Jacquet functors do not create new supercuspidal support: the supercuspidal
support of every Jacquet constituent is obtained from the supercuspidal
support of the original representation.

Because $\sigma$ is tempered, every discrete-series factor on the
$\rho$-line is centered, say $\delta_K(\rho)$, with ordinary support
\[
   \{\nu^{1-K}\rho,\nu^{3-K}\rho,\ldots,\nu^{K-1}\rho\}.
\]
Thus, if $\nu^{-L}\rho$ occurs in the support of such a factor, symmetry of
the centered segment implies that $\nu^{L}\rho$ occurs in the same factor.
Equivalently, $K-1\geq L$, and hence $K\geq L+1$. This contradicts the
maximality of $L$.

For the closed orbit $r=Ld$, the whole factor $\delta_L(\rho)$ lies on the
diagonal $G_{Ld}$-part of the stabilizer. The good-orbit pairing condition
would therefore force
$
   \delta_L(\rho)\simeq \nu T,
   \qquad\text{hence}\qquad
   T\simeq\nu^{-1}\delta_L(\rho),
$
which has just been shown to be impossible. Thus the closed orbit cannot
contribute.
\end{proof}

\subsection{The contributing orbit}

\begin{proposition}[The contributing orbit]
\label{prop:unique-orbit}
Let $
   \pi=\delta_L(\rho)\times\sigma
$
be irreducible tempered and $H_n$-distinguished. Assume that
$\rho\in\cD_d$, that $L\ge2$, and that $L$ is maximal among the lengths on
the $\rho$-line. Then every good orbit in the maximal-parabolic Mackey
filtration for $P_{Ld,N}$ has
\[
 r=(L-1)d.
\]
For a contributing irreducible subquotient
$S\otimes T$ of
$r_{(N-(L-1)d,(L-1)d),(N)}(\sigma)$ one has
$
  T\simeq\delta_{L-1}(\rho).
$
Consequently $\sigma$ contains a whole original factor
$\delta_{L-1}(\rho)$.
\end{proposition}

\begin{proof}
Since $\rho$ is distinguished, $s(\rho)=2$.

\emph{Step 1: exclude the open orbit and determine the interior orbit.}
The open orbit $r=0$ is not good, because it would require
$\delta_L(\rho)$ itself to be distinguished; this is impossible for
$L>1$ by Theorem~\ref{thm:ds}. Now let $r=td$ with $0<t<L$. By
Lemma~\ref{lem:maximal-factor-jacquet} and
\eqref{eq:maximal-three-conditions}, the fixed symplectic factor
\[
   \nu^{-t}\delta_{L-t}(\rho)
\]
must be distinguished. Since $\nu$ is trivial on the symplectic group,
this is equivalent to distinction of $\delta_{L-t}(\rho)$. Theorem
\ref{thm:ds} therefore forces
\[
   L-t=1.
\]
Hence the only possible interior good orbit is
$
  r=(L-1)d.
$

\emph{Step 2: exclude the full-pairing closed orbit.}
If $r=Ld$ is available, the whole factor $\delta_L(\rho)$ lies on the
diagonal $G_{Ld}$-part of the stabilizer. The diagonal pairing would force
\[
   \delta_L(\rho)\simeq\nu T,
   \qquad\text{hence}\qquad
   T\simeq\nu^{-1}\delta_L(\rho).
\]
Lemma~\ref{lem:closed-orbit-maximality} proves that this is impossible: it
would put $\nu^{-L}\rho$ in the cuspidal support of $\sigma$, and
centeredness of the tempered $\rho$-line factors would then force a factor
of length at least $L+1$, contradicting maximality. Thus $r=Ld$ does not
contribute.

\emph{Step 3: identify the terminal factor for the unique orbit.}
For $t=L-1$, Lemma~\ref{lem:maximal-factor-jacquet} gives
\begin{equation}\label{eq:unique-jacquet}
   r_{((L-1)d,d),(Ld)}\bigl(\delta_L(\rho)\bigr)
   =
   \nu\delta_{L-1}(\rho)
   \otimes
   \nu^{-(L-1)}\rho.
\end{equation}
Let $S\otimes T$ be a contributing irreducible subquotient of
$
   r_{(N-(L-1)d,(L-1)d),(N)}(\sigma).
$
The diagonal $G_{(L-1)d}$-condition in
\eqref{eq:maximal-three-conditions} is
$
   \nu\delta_{L-1}(\rho)\simeq\nu T,
$
and therefore
$
  T\simeq\delta_{L-1}(\rho).
$

By endpoint rigidity, namely Remark~\ref{rem:adjacent-terminal-rigidity},
this $T$ is contributed by one whole original
$\delta_{L-1}(\rho)$-factor of $\sigma$ and by no other factor. In particular, $\sigma$ contains a whole original
factor $\delta_{L-1}(\rho)$.
\end{proof}

\subsection{Rank reduction}

\begin{theorem}[Rank reduction]
\label{thm:rank-reduction}
Let $
   \pi=\delta_L(\rho)\times\sigma
$
be irreducible tempered and $H_n$-distinguished, and suppose $L$ is maximal
among the lengths occurring on the $\rho$-line. Then $\rho$ is
distinguished. Moreover:
\begin{enumerate}[label=\textup{(\roman*)}]
\item if $L=1$, then $\sigma$ is distinguished;
\item if $L\ge2$, then, after reordering the tempered factors,
\[
   \sigma\simeq\delta_{L-1}(\rho)\times\pi_0
\]
with $
   \Hom_{H_{n-(2L-1)d}}(\pi_0,\C)\neq0.
$

Equivalently, $
   \pi\simeq\cB_{L-1}(\rho)\times\pi_0
$
with $\pi_0$ tempered and distinguished.
\end{enumerate}
\end{theorem}

\begin{proof}
Proposition~\ref{prop:distinguished-cuspidal-support} first shows that
$\rho$ is distinguished.

Suppose $L=1$. The only possible good orbit is the open orbit. Indeed, if
the full-pairing orbit is available, it would force a terminal constituent
$\nu^{-1}\rho$, excluded by Lemma~\ref{lem:closed-orbit-maximality}; if it
is not available, there is nothing to exclude. The open-orbit formula is
\[
   \Hom_{H_d}(\rho,\C)
   \otimes
   \Hom_{H_{n-d}}(\sigma,\C),
\]
and the first factor is nonzero. Hence $\sigma$ is distinguished.

Now assume $L\ge2$. We use the rank-reduction argument
\eqref{eq:rank-reduction-mechanism} step by step.

\emph{Step 1: the unique good orbit.}
By Proposition~\ref{prop:unique-orbit}, every good orbit has
\[ r=(L-1)d.
\]
The proof of that proposition includes the full exclusion of the closed
orbit $r=Ld$: its pairing would force
$\nu^{-1}\delta_L(\rho)$ in a Jacquet module of $\sigma$, hence
$\nu^{-L}\rho$ in the cuspidal support of $\sigma$, which by centeredness
would force a $\rho$-line factor of length at least $L+1$, contrary to the
maximality of $L$.

\emph{Step 2: the adjacent discrete-series factor occurs in the Jacquet
module of $\sigma$.}
Let
$
   S\otimes T
$
be a constituent of
\[
   r_{(n_0,(L-1)d)}(\sigma),
   \qquad n_0=n-(2L-1)d,
\]
selected by the good orbit. From \eqref{eq:unique-jacquet}, the
$G_{(L-1)d}$-part coming from $\delta_L(\rho)$ is
$\nu\delta_{L-1}(\rho)$. Hence the diagonal pairing gives
\[
   \nu\delta_{L-1}(\rho)\simeq\nu T,
\]
so
$ T\simeq\delta_{L-1}(\rho).
$

\emph{Step 3: $\delta_{L-1}(\rho)$ is a whole original factor of
$\sigma$.}
The ordinary $\nu$-support of $T=\delta_{L-1}(\rho)$ is
\[
   2-L,\ 4-L,\ldots,\ L-2.
\]
Choose an original $\rho$-line factor $\delta_K(\rho)$ of $\sigma$ whose
terminal Jacquet contribution supplies the top exponent $L-2$. Such a
terminal piece begins at $1-K$. Since all its exponents must lie in the
support of $T$, we have
\[
   1-K\ge2-L,
\]
so $K\le L-1$. On the other hand, in order to contain the top exponent
$L-2$ one must have
\[
   K-1\ge L-2,
\]
so $K\ge L-1$. Therefore
$
   K=L-1.
$

A proper lower terminal piece of $\delta_{L-1}(\rho)$ stops before the top
endpoint $L-2$. Hence this factor must contribute its whole centered
segment to $T$. Since its rank is already $(L-1)d$, the full rank of $T$,
no other original factor contributes to the terminal block. Thus, after
reordering,
\[  \sigma\simeq\delta_{L-1}(\rho)\times\pi_0
   \]
for a tempered representation $\pi_0\in\Irr(G_{n_0})$, and the complementary
Jacquet factor is
$ S\simeq\pi_0.
$

This is precisely the endpoint-rigidity observation of
Remark~\ref{rem:adjacent-terminal-rigidity} applied to the present good
orbit.

\emph{Step 4: the remainder $\pi_0$ is distinguished.}
For $r=(L-1)d$, the Levi part of the orbit stabilizer is
\[
   \Delta\bigl(G_{(L-1)d}\times G_{(L-1)d}\bigr)
   \times H_d\times H_{n_0}.
\]
The three good-orbit conditions are therefore
\[
   \nu\delta_{L-1}(\rho)\simeq\nu T,
   \qquad
   \Hom_{H_d}(\nu^{-(L-1)}\rho,\C)\neq0,
   \qquad
   \Hom_{H_{n_0}}(S,\C)\neq0.
\]
The first condition has already identified $T$. The second is nonzero
because $\rho$ is distinguished and $\nu|_{H_d}=1$. Since the orbit is
good, the third condition is also nonzero. Using $S\simeq\pi_0$ gives
$ \Hom_{H_{n_0}}(\pi_0,\C)\neq0.
$
Thus
\[
   \pi\simeq
   \delta_L(\rho)\times\delta_{L-1}(\rho)\times\pi_0
   =\cB_{L-1}(\rho)\times\pi_0,
\]
with $\pi_0$ tempered and distinguished, completing the rank reduction.
\end{proof}

\begin{corollary}[Rank reduction for an elementary block]
\label{thm:block-rank-reduction}
Let $\rho\in\cD_d$, let $m\ge1$, and let
$\pi_0\in\Irr(G_{n_0})$ be tempered. Put
\[
   \Pi
   =\cB_m(\rho)\times\pi_0
   =\delta_{m+1}(\rho)\times\delta_m(\rho)\times\pi_0,
\]
and assume that $\delta_{m+1}(\rho)$ has maximal length among the factors
on the $\rho$-line. If $\Pi$ is $H_n$-distinguished, then $\pi_0$ is
$H_{n_0}$-distinguished.
\end{corollary}

\begin{proof}
Apply Theorem~\ref{thm:rank-reduction} with $L=m+1$ and
$\sigma=\delta_m(\rho)\times\pi_0$. The theorem removes the elementary
block $\cB_m(\rho)$ and gives
\[
   \Hom_{H_{n_0}}(\pi_0,\C)\neq0.
\]
\end{proof}

\subsection{Induction and the converse}

\begin{theorem}[General converse]\label{thm:general-converse}
Let $\pi\in\Irr(G_n)$ be  an irreducible tempered representation. If
$
   \Hom_{H_n}(\pi,\C)\neq0,
$
then
\[ \pi
   \simeq
   \cB_{m_1}(\rho_1)
   \times\cdots\times
   \cB_{m_s}(\rho_s),\]
where $\rho_i\in\cD$ and $m_i\ge0$.
\end{theorem}

\begin{proof}
We induct on the number of unitary discrete-series factors in the tempered
support of $\pi$. Choose a supercuspidal line and a factor
$\delta_L(\rho)$ of maximal length on that line, and place it first.
Proposition~\ref{prop:distinguished-cuspidal-support} shows that $\rho$ is
distinguished, and Theorem~\ref{thm:rank-reduction} gives the corresponding
reduction.

If $L=1$, Theorem~\ref{thm:rank-reduction}\textup{(i)} gives
$ \pi\simeq\cB_0(\rho)\times\pi_0
$
with $\pi_0$ tempered and distinguished. If $L\ge2$, part \textup{(ii)}
gives
$ \pi\simeq\cB_{L-1}(\rho)\times\pi_0
$
with $\pi_0$ again tempered and distinguished. In either case the number
of discrete-series factors strictly decreases. Applying the induction
hypothesis to $\pi_0$ gives the required product of elementary blocks.
\end{proof}

The existence theorem and the converse completes the proof of the Theorem~\ref{conj:tempered-classification}  and provides the classification of temepered distinguished representations.

\begin{theorem}[Prasad's revised tempered conjecture]
\label{thm:main-classification}
Let $\pi\in\Irr(G_n)$ be an irreducible tempered representation. Then
\[
   \Hom_{H_n}(\pi,\C)\neq0
\]
if and only if
\[  \pi
   \simeq
   \cB_{m_1}(\rho_1)
   \times\cdots\times
   \cB_{m_s}(\rho_s),
   \]
with $\rho_i\in\cD$ and $m_i\ge0$.
\end{theorem}

\begin{proof}
The ``only if'' direction is Theorem~\ref{thm:general-converse}. For the
``if'' direction, every $\cB_m(\rho)$ is distinguished by
Theorem~\ref{thm:Bm-general}, with $m=0$ interpreted as a distinguished
supercuspidal singleton, and arbitrary products are distinguished by the
hereditary property.
\end{proof}


\begin{thebibliography}{99}

\bibitem{ABPS2016}
A.-M. Aubert, P. Baum, R. Plymen and M. Solleveld,
\emph{Depth and the local Langlands correspondence},
in \emph{Arbeitstagung Bonn 2013},
Progress in Mathematics, Vol.~319,
Birkh\"auser/Springer, Cham, 2016, pp.~17--41.

\bibitem{HeumosRallis1990}
M.~J.~Heumos and S.~Rallis,
\emph{Symplectic--Whittaker models for $\GL_n$},
Pacific J. Math. {146} (1990), no.~2, 247--279.

\bibitem{OffenSayag2007}
O.~Offen and E.~Sayag,
\emph{On unitary representations of $\GL_{2n}$ distinguished by the
symplectic group},
J. Number Theory {125} (2007), no.~2, 344--355.

\bibitem{OffenSayag2008Global}
O.~Offen and E.~Sayag,
\emph{Global mixed periods and local Klyachko models for the general linear
group},
Int. Math. Res. Not. IMRN (2008), no.~1, 25 pp.

\bibitem{OffenSayag2008Uniqueness}
O.~Offen and E.~Sayag,
\emph{Uniqueness and disjointness of Klyachko models},
J. Funct. Anal. {254} (2008), no.~11, 2846--2865.


\bibitem{dprasad}
D.~Prasad,
\emph{Email correspondence with M.~K.~Verma concerning tempered
representations of $\GL_n(D)$ with symplectic period},
August 2026.


\bibitem{BernsteinZelevinsky1977}
I.~N.~Bernstein and A.~V.~Zelevinsky,
\emph{Induced representations of reductive $p$-adic groups. I},
Ann. Sci. \'Ecole Norm. Sup. (4) {10} (1977), 441--472.

\bibitem{DeligneKazhdanVigneras1984}
P.~Deligne, D.~Kazhdan and M.-F.~Vign\'eras,
\emph{Repr\'esentations des alg\`ebres centrales simples $p$-adiques},
in \emph{Representations of reductive groups over a local field},
Travaux en Cours, Hermann, Paris, 1984, 33--117.

\bibitem{MatringeOffenYang2024}
N.~Matringe, O.~Offen and C.~Yang,
\emph{On local intertwining periods},
J. Funct. Anal. {286} (2024), no.~4, 110293.

\bibitem{Mitra2014}
A.~Mitra,
\emph{On representations of $\GL_{2n}(F)$ with a symplectic period},
Pacific J. Math. {268} (2014), 435--463.

\bibitem{MitraOffenSayag2017}
A.~Mitra, O.~Offen and E.~Sayag,
\emph{Klyachko models for ladder representations},
Doc. Math. {22} (2017), 611--657.

\bibitem{LapidMinguez2016}
E.~Lapid and A.~M\'inguez,
\emph{On parabolic induction on inner forms of the general linear group over
a non-archimedean local field},
Selecta Math. (N.S.) {22} (2016), no.~4, 2347--2400.


\bibitem{Matringe2026}
N.~Matringe,
\emph{Cuspidal distinction via functional equation of intertwining periods},
preprint, 2026.


\bibitem{Tadic2015}
M.~Tadi\'c,
\emph{Remark on representation theory of general linear groups over a
	non-archimedean local division algebra},
Rad Hrvatske Akademije Znanosti i Umjetnosti, Matemati\v{c}ke Znanosti
{19 (523)} (2015), 27--53.


\bibitem{MSSS2026}
N.~Matringe, V.~S\'echerre, S.~Stevens and M.~Suzuki,
\emph{Discrete series representations of quaternionic $\GL_n(D)$ with
symplectic periods},
to appear in J. Reine Angew. Math.; arXiv:2503.08955.

\bibitem{LapidMinguez2014}
E.~Lapid and A.~M\'inguez,
\emph{On a determinantal formula of Tadi\'c},
Amer. J. Math. {136} (2014), 111--142.

\bibitem{Offen2017}
O.~Offen,
\emph{On parabolic induction associated with a $p$-adic symmetric space},
J. Number Theory {170} (2017), 211--227.


\bibitem{SharmaZelevinsky2026}
H.~Sharma,
\emph{Symplectic model for Zelevinsky modules of $\GL(n,D)$},
Algebr. Represent. Theory (2026).

\bibitem{SharmaVerma2025}
H.~Sharma and M.~K.~Verma,
\emph{Symplectic period for a representation of $\GL_n(D)$},
J. Pure Appl. Algebra {229} (2025), 107976.

\bibitem{SharmaVerma2026}
H.~Sharma and M.~K.~Verma,
\emph{Symplectic model for ladder and unitary representations},
J. Pure Appl. Algebra {230} (2026), 108334.

\bibitem{Tadic1990}
M.~Tadi\'c,
\emph{Induced representations of $\GL(n,A)$ for $p$-adic division algebras
$A$},
J. Reine Angew. Math. {405} (1990), 48--77.

\bibitem{Verma2018}
M.~K.~Verma,
\emph{On symplectic periods for inner forms of $\GL_n$},
Math. Res. Lett. \textbf{25} (2018), no.~1, 309--334.

\end{thebibliography}
\end{document}